\documentclass[runningheads]{llncs}
\usepackage[T1]{fontenc}
\usepackage{graphicx}
\usepackage{amsmath,amssymb}
\usepackage{color}
\usepackage{tikz}
\usetikzlibrary{arrows.meta, positioning}

\newcommand{\ones}{\underline{1}}
\newcommand{\E}{{\mathbb E}}
\newcommand{\Prob}{{\mathbb P}}

\begin{document}
\title{Equilibrium Joining Strategies for Queues\\ 
in Two-Phase Random Environment\thanks{This paper has been accepted at 
the 29th International Conference on Analytical and Stochastic Modelling 
Techniques and Applications (ASMTA 2026) with proceedings published in 
Springer LNCS.}}
\titlerunning{Equilibrium Joining Strategies for Queues in Random Environment}
\author{Konstantin Avrachenkov\inst{1} \and Uri Yechiali\inst{2}}
%
% without ORCID ids:
% \author{Konstantin Avrachenkov\inst{1} \and Uri Yechiali\inst{2}}
% with ORCID ids:
% \author{Konstantin Avrachenkov\inst{1}\orcidID{0000-0002-8124-8272)} \and
% Uri Yechiali\inst{2}\orcidID{0000-0002-6644-411X} }
%
\authorrunning{K. Avrachenkov and U. Yechiali}
\institute{Inria Sophia Antipolis, France\\
\email{k.avrachenkov@inria.fr} \and
Tel Aviv University, Israel\\
\email{uriy@tauex.tau.ac.il}}
\maketitle              % typeset the header of the contribution
\begin{abstract}
We study equilibrium joining strategies in an M/M/1-type queueing system with strategic customers operating 
in a two-phase random environment described as a continuous-time Markov process. Strategic customers, upon arrival, 
choose whether to join or to balk based on available information and anticipated utility, considering the trade-off 
between reward from service and waiting cost.
Four observation scenarios are analysed: fully observable (both queue length and environment phase are disclosed to 
a customer upon arrival), queue-only observable, environment-only observable, and fully unobservable. In each case, 
equilibrium joining strategies are analysed. In the environment-only observable and fully unobservable cases, explicit 
solutions and equilibrium conditions are derived under rapid oscillations and under very slow transitions between 
environment phases.

\end{abstract}

\section{Introduction}

We consider an M/M/1-type queue in a random environment modeled by a two-phase continuous-time Markov chain.
Specifically, when the environment is in phase $i=1,2$ the arrival rate is $\lambda_i$, the service rate is $\mu_i$, 
and the environment remains in that phase for an exponential time with parameter $\eta_i$, after which it switches to the other phase. 
%Thus, the rate matrix of the underlying environment process is 
%$$
%A = \left[\begin{array}{rr}
%-\eta_1 & \eta_1 \\
%\eta_2 & -\eta_2
%\end{array}\right].
%$$
This queueing system has been extensively analysed (see e.g., \cite{YechialiNaor1971,Purdue1974,Neuts1978,BaccelliMakowski1986,Gupta2006}
and follow-up publications). The focus of our work is on the analysis of equilibrium 
joining strategies of individual customers. 
We assume that all the parameters of the model constitute common knowledge.
We consider various cases of available information regarding the environment phase and the queue length for arriving customers. 
The following four cases will be investigated: (a) fully observable (FO) case, when both the environment phase and the queue length are known upon arrival; (b) the case when only the queue length is observable (QO) upon arrival; (c) the case when only the environment phase
is observable (EO) upon arrival and, finally, (d) no observation (NO) case when neither the environment phase nor the queue length are observable. 

Inspired by the seminal article of Naor \cite{Naor} (see also \cite{Yechiali1971,Yechiali1972,Edelson1975}), we would like 
to analyse equilibrium joining strategies in an evolving environment. Upon arrival, 
a customer can decide whether to join or to balk. If a customer balks, they leave the system forever. If a customer joins, they gain an eventual reward $R$ for the service and accumulate waiting cost at a rate $c_w$ per unit of time.
Customers who join cannot subsequently renege. Reneging may be beneficial when the environment switches from a fast to a slow service phase. Its value in queueing systems 
with server vacations or failures is studied in \cite{Economou2022}.

In \cite{Economou2013} equilibrium joining strategies have been studied
in a system with random environment when the server removes all present
customers at the completion epochs of exponential catastrophe cycles. See also \cite{Paz2007} and \cite{Paz2014}, where related models of $M/M/\infty$ and $M/M/1$, respectively, with disasters have been analysed. In \cite{Economou2016}, strategic queueing has been studied 
for fluid queues in random environment. In \cite{Hassinatal2023}
strategic queueing with arrival process with an i.i.d. random rate has been analysed. 
The special case $\lambda_1=\lambda_2$ and $\mu_1=0$ corresponds to an $M/M/1$ queue 
with an unreliable server. Equilibrium balking in this setting is studied in \cite{EconomouKanta2008}. 
A model related to this special case arises in the dedicated-vs-free spectrum choice in wireless networks, 
as studied in \cite{Jagannathan2012}.
For background on strategic queueing we recommend 
 the books \cite{HassinHaviv,Hassin2016}.

\section{FO case: fully observable system}

We start our analysis with the fully observable (FO) case.
Let $S_i(n)$ denote the sojourn time of a customer joining a queue of length $n$, when the environment is in phase $i=1,2$. 
Let $s_i(n)$ be its expected value. In particular, $S_i(0)$ is the customer's service time given the service starts in 
environment $i=1,2$. 
Using the memoryless property of the exponential distribution, we can write the following
equation:
\begin{equation}
\label{eq:sys4si}
s_i(0) = \frac{1}{\mu_i+\eta_i} + \frac{\eta_i}{\mu_i+\eta_i}\,s_j(0),\quad i\neq j,\quad i,j\in\{1,2\}.
\end{equation}
That is, $1/(\mu_i+\eta_i)$ is the expected time until either the service is concluded or the environment 
changes phase, and then, with probability $\eta_i/(\mu_i+\eta_i)$ the environment switches to phase $j$, 
and the remaining expected service time is $s_j(0)$.
Solving the system (\ref{eq:sys4si}), we obtain
\begin{equation}
\label{eq:si}
s_i(0) = \frac{\mu_j+\eta_1+\eta_2}{\mu_1\mu_2+\mu_1\eta_2+\mu_2\eta_1},
\quad i\neq j,\quad i,j\in\{1,2\}.
\end{equation}

We also define the probability that a service initiated in environment $i$ is completed in environment $j$:
$$
q_{ij} = \mathbb{P}[\text{service completion occurs in environment } j \mid \text{service has started in } i].
$$
Reasoning as in the derivation of (\ref{eq:sys4si}) gives:
$$
q_{ii} = \frac{\mu_i}{\mu_i+\eta_i} + \frac{\eta_i}{\mu_i+\eta_i}\,q_{ji},
\quad i\neq j,\quad i,j\in\{1,2\},
$$
The solution of the above linear system results in
\begin{equation}
\label{eq:qii}
q_{ii} = \frac{\mu_i(\mu_j+\eta_j)}{\mu_1\mu_2+\mu_1\eta_2+\mu_2\eta_1}, \quad i\neq j,\quad i,j\in\{1,2\},
\end{equation}
and we have $q_{ij}=1-q_{ii}$ for $j\neq i$. Equation (\ref{eq:qii}) can be interpreted as follows:
$$
q_{ii}
= \sum_{k=0}^\infty \left( \frac{\eta_1\eta_2}{(\mu_1+\eta_1)(\mu_2+\eta_2)}\right)^k \frac{\mu_i}{\mu_i+\eta_i},\quad i=1,2,
$$
%  = \frac{\frac{\mu_i}{\mu_i+\eta_i}}{1-\frac{\eta_1\eta_2}{(\mu_1+\eta_1)(\mu_2+\eta_2)}}
that is, the underlying process performs $k$ phase transitions back and forth, and then the service 
completion occurs before the phase transition.

%Next, we define the conditional expected sojourn time
%$$
%s_i(n)=\mathbb{E}[\text{sojourn time} \mid \text{arrival sees queue length } n \text{ and environment in phase } i].
%$$
When a customer arrives and sees $n$ customers ahead, the first service (of the customer in service) will take an expected time $s_i(0)$ if the environment is in phase $i$. At the end of this service, the environment is in phase $i$ with probability $q_{ii}$ and in phase $j$ with probability $q_{ij}$ (where $j\neq i$). Hence, the recursion for the conditional expected sojourn time is:
$$
s_i(n)=s_i(0)+q_{ii}\,s_i(n-1)+q_{ij}\,s_j(n-1),\quad i\neq j.
$$
Define the vectors
$$
s(n)=\left[\begin{array}{cc}s_1(n)\\s_2(n)\end{array}\right],\quad n=0,1,...
$$
and the matrix
$$
Q=\left[\begin{array}{cc}
q_{11} & q_{12} \\
q_{21} & q_{22}
\end{array}\right]
=\left[\begin{array}{cc}
q_{11} & 1-q_{11} \\
1-q_{22} & q_{22}
\end{array}\right].
$$
Then, the recurrence becomes
\begin{equation}
\label{eq:Qrecur}
s(n) = s(0) + Q\,s(n-1).
\end{equation}
This recurrence leads to closed-form expressions for $s(n)$.

\begin{proposition}
\label{prop:Qclosed}
The vector of the expected conditional sojourn times has
the following closed-form expressions:
\begin{equation}
\label{eq:sn_1}
s(n) = \sum_{k=0}^{n}Q^k s(0)
\end{equation}
and
\begin{equation}
\label{eq:sn_2}
s(n) = [(n+1) \ones \pi + H(I-Q^{n+1})] s(0),
\end{equation}
where 
$\pi=\left[\frac{\mu_1\eta_2}{\mu_1\eta_2+\mu_2\eta_1} \quad
\frac{\mu_2\eta_1}{\mu_1\eta_2+\mu_2\eta_1}\right]$ 
is the invariant probability measure of $Q$
and $H$ is the deviation matrix.
\end{proposition}
\noindent {\bf Proof:}
By iterating the recurrence (\ref{eq:Qrecur}), we obtain the closed-form expression
(\ref{eq:sn_1}).

The other expression can be obtained using the deviation matrix \cite{KemenySnell1976}
$$
H=(I-\ones \pi)(I-Q+\ones \pi)^{-1},
$$
where $\pi$ is the invariant probability measure of $Q$, i.e., $\pi = \pi Q$ with $\pi \ones = 1$.
Note that $\pi$ is a row vector and $\ones$ is a column vector of 1s.
We can write
$$
(I-Q) \sum_{k=0}^{n} Q^k = I - Q^{n+1}.
$$
Premultiplying the above equality by the matrix $H$ and using the property of the deviation matrix $H(I-Q)=I-\ones \pi$ (see e.g., the Appendix in \cite{KemenySnell1976}), we obtain
$$
(I-\ones \pi) \sum_{k=0}^{n} Q^k = H(I - Q^{n+1}).
$$
Then, using $\pi = \pi Q$ yields
$$
\sum_{k=0}^{n} Q^k = [(n+1) \ones \pi + H(I-Q^{n+1})]. 
$$
Postmultiplication of the above equation by $s(0)$ leads to (\ref{eq:sn_2}).
\hfill $\Box$

\medskip

We can also establish the following monotonicity result.
\begin{lemma}
\label{lem:monotone1}The following inequality holds
\begin{equation}
\label{eq:sojourn-level-monotonicity}
s_i(n+1)>s_i(n), \qquad i\in\{1,2\}.
\end{equation}
\end{lemma}
\noindent {\bf Proof:}
From (\ref{eq:sn_1}), we have
$$
s(n+1)-s(n)=Q^{n+1}s(0).
$$
Since $Q$ is stochastic and $s(0)$ has strictly positive components,
the inequality (\ref{eq:sojourn-level-monotonicity}) follows.
\hfill $\Box$

\medskip

When an arriving customer sees the environment in phase $i$ and $n$ customers
in the queue, the customer does not incur a loss and joins the queue if
\begin{equation}
\label{eq:ind_strategy}
R \ge c_w \, s_i(n), \quad i=1,2.
\end{equation}
Recall that $R$ is the reward for the service and $c_w$ is the rate
of the waiting cost.
We assume that for someone who has already arrived at the system, it makes sense 
to join the queue even if the expected utility is zero.
Thus, by Lemma~\ref{lem:monotone1}, the equilibrium strategy has
phase-dependent threshold structure with the thresholds
$$
\bar{n}_i = \min\left\{ n \in {\mathbb N}_0 : s_i(n) > \frac{R}{c_w} \right\}, \quad i=1,2,
$$
i.e., the first queue length at which the customer balks.

\section{QO case: only the queue length is observable}

Now we consider the QO case when the queue length is observable upon arrival but not the phase
of the environment. An arriving customer joins the queue if their expected
utility is non-negative. In the QO case, the expected utility is given by
$$
U_{L=n} = R - c_w \E^0[S|L=n], 
$$
where $S$ is the sojourn time, $L$ is the queue length observed upon arrival, and
$\E^0[\cdot]$ denotes Palm expectation, i.e., the expectation given an arrival event
(for more background on Palm expectation and probability, see e.g. \cite{BaccelliBremaud1987}). 
Then, conditioning on the environment phase $E$, we can write
$$
U_{L=n} =
$$
$$
R - c_w (\E^0[S|E=1,L=n]\Prob^0[E=1|L=n]+\E^0[S|E=2,L=n]\Prob^0[E=2|L=n])
$$
where the operations with zero superscript correspond
to Palm versions. Specifically, the Palm probabilities are given by
\begin{equation}
\label{eq:P0_in}
\Prob^0[E=i,L=n] = \frac{\lambda_i p_{in}}{\lambda_1p_{1\bullet}+\lambda_2p_{2\bullet}},
\quad i=1,2, \quad n=0,1,... 
\end{equation}
where $p_{in}=\Prob[E=i,L=n]$ are the stationary probabilities and
$p_{i\bullet} = \sum_{n=0}^{\infty} p_{in} = \eta_j/(\eta_1+\eta_2)$, $i=1,2,$ $j\neq i$.
Hence,
$$
\Prob^0[E=i|L=n] = \frac{\lambda_i p_{in}}{\lambda_1p_{1n}+\lambda_2p_{2n}},
\quad i=1,2, \quad n=0,1,... 
$$
Thus, the expected utility can be written as follows:
\begin{equation}
\label{eq:U_QOcase}
U_{L=n} = R - c_w \frac{1}{\lambda_1 p_{1n}+\lambda_2 p_{2n}}( \lambda_1 p_{1n} s_1(n)
+ \lambda_2 p_{2n} s_2(n) ).
\end{equation}

In general, a symmetric equilibrium in the QO case can have the form 
of an infinite vector of joining probabilities, i.e., $a^*=(a_0^*,a_1^*,...)$.
Next, we show that under natural conditions the equilibrium strategy
has a threshold structure.

\begin{proposition}
\label{prop:QO-monotonicity}
Suppose that $\mu_1\leq \mu_2$ (without loss of generality). 
Fix a stationary joining strategy
$a=(a_n)_{n\geq 0}$, and let $p^a_{in}$ denote the corresponding 
stationary probabilities. 
Assume that
\begin{equation}
\label{eq:MLR-strategy-a}
    p^a_{1,n+1}p^a_{2n}
    \geq
    p^a_{1n}p^a_{2,n+1}
\end{equation}
whenever both $n$ and $n+1$ have positive Palm probability. Then,
the utility $U_n^a$ to join when seeing the queue length $n$ and everyone
else follows strategy $a$ is strictly decreasing in $n$ on the support of the Palm distribution.
\end{proposition}

\noindent {\bf Proof:}
First note that from (\ref{eq:si}) we have
\begin{equation}
\label{eq:s0diff}
    s_1(0)-s_2(0) =
    \frac{\mu_2-\mu_1}
    {\mu_1\mu_2+\mu_1\eta_2+\mu_2\eta_1}
    \geq 0.
\end{equation}
Also note that from (\ref{eq:qii}) for every $x=(x_1,x_2)^\top$,
\begin{equation}
\label{eq:Qxdiff}
    (Qx)_1-(Qx)_2 = \delta(x_1-x_2),
\end{equation}
where
$$
    \delta :=
    q_{11}+q_{22}-1
    =
    \frac{\mu_1\mu_2}
    {\mu_1\mu_2+\mu_1\eta_2+\mu_2\eta_1}
    \in(0,1).
$$
It follows from (\ref{eq:sn_1}), (\ref{eq:s0diff}) and (\ref{eq:Qxdiff}) that
\begin{equation}
\label{eq:sojourn-phase-order}
    s_1(n)-s_2(n) =
    \bigl(s_1(0)-s_2(0)\bigr)
    \sum_{k=0}^n\delta^k
    \geq 0.
\end{equation}

For every $n$ of positive Palm probability, set
$$
\alpha_n^a :=
    \mathbb{P}_a^0(E=1 | L=n)
    =
    \frac{\lambda_1p^a_{1n}}
    {\lambda_1p^a_{1n}+\lambda_2p^a_{2n}}.
$$
Condition (\ref{eq:MLR-strategy-a}) implies
\begin{equation}
\label{eq:posterior-monotonicity}
    \alpha_{n+1}^a\geq \alpha_n^a.
\end{equation}
Indeed, after cross-multiplication, (\ref{eq:posterior-monotonicity})
is equivalent to
$$
    \lambda_1\lambda_2
    \left(
        p^a_{1,n+1}p^a_{2n}
        -p^a_{1n}p^a_{2,n+1}
    \right)
    \geq 0.
$$
For notational convenience, denote
$$
    \sigma_a(n):=
    \mathbb{E}_a^0[S | L=n]
    =\alpha_n^a s_1(n)
    +(1-\alpha_n^a)s_2(n).
$$
Then, we can write
\begin{align*}
    \sigma_a(n+1)-\sigma_a(n)
    ={}&
    \alpha_{n+1}^a
    \bigl(s_1(n+1)-s_1(n)\bigr)
    \\
    &+
    (1-\alpha_{n+1}^a)
    \bigl(s_2(n+1)-s_2(n)\bigr)
    \\
    &+
    (\alpha_{n+1}^a-\alpha_n^a)
    \bigl(s_1(n)-s_2(n)\bigr).
\end{align*}
The first two terms are strictly positive by Lemma~\ref{lem:monotone1}, whereas the last term is
nonnegative by \eqref{eq:sojourn-phase-order} and
\eqref{eq:posterior-monotonicity}. Thus
$$
    \sigma_a(n+1) > \sigma_a(n).
$$
Since $U_n^a=R-c_w \sigma_a(n)$, the sequence $U_n^a$ is strictly
decreasing. \hfill $\Box$

\medskip 

If a stationary strategy $a$ induces a stationary distribution 
satisfying (\ref{eq:MLR-strategy-a}), then its joining utility is decreasing in $n$. 
In particular, if an equilibrium strategy $a^*$ induces a distribution satisfying (\ref{eq:MLR-strategy-a}), then, on the support of the Palm distribution, $a^*$ has a threshold form. Yet, there may be non-threshold equilibrium strategies.

The condition (\ref{eq:MLR-strategy-a}) has a probabilistic interpretation in terms of the 
likelihood ratio stochastic order \cite{MullerStoyan2002}: 
$L \mid E=1 \ge_{lr} L \mid E=2$. This means that the likelihood 
ratio
$$ 
\frac{{\mathbb P}(L=n | E=1)}{{\mathbb P}(L=n | E=2)} 
$$
is nondecreasing in $n$. Hence, observing a larger queue provides increasingly stronger evidence that the system is in the slower phase~1.

We have an explicit expression for $s_i(n)$, equation (\ref{eq:sn_2}). However, 
it is unfortunate that there is no explicit expression for the stationary probability $p_{in}$.
Therefore, as in \cite{YechialiNaor1971}, we consider two limiting cases:
the case of rapid oscillations of the environment (Case C in \cite{YechialiNaor1971})
and the case of very slow transitions between the environment phases 
(Case D in \cite{YechialiNaor1971}). In these cases, the stationary distribution
has an explicit product-form.

%Based on \cite{Yechiali1971,Yechiali1972} and \cite{Economou2013}, we conjecture in this study that equilibrium
%strategies in the QO case are of threshold type. We leave a thorough investigation of this conjecture 
%as a future research.

\subsection{QO case: rapid oscillations of the environment}

Here, as in \cite{YechialiNaor1971}, we assume that the environment oscillates rapidly between 
two phases 1 and 2. Specifically, $\eta_1$ and $\eta_2$ tend to infinity simultaneously and their
ratio $\eta_1/\eta_2$ tends to a finite constant $C$. In the limit, as the rates
go to infinity, the stationary probabilities converge to
\begin{equation}
\label{eq:stationaryQOrapid}
p_{in} = p_{i\bullet} \frac{1-\bar{\rho}}{1-\bar{\rho}^{\bar{n}+1}} \bar{\rho}^n,
\quad n=0,...,\bar{n},
\end{equation}
where $\bar{n}$ is the equilibrium threshold and $\bar{\rho} = \bar{\lambda}/\bar{\mu}$ with
$$
\bar{\lambda} = p_{1\bullet} \lambda_1 + p_{2\bullet}  \lambda_2,
\quad 
\bar{\mu} = p_{1\bullet} \mu_1 + p_{2\bullet} \mu_2,
$$
and $p_{1\bullet} = 1/(1+C)$, $p_{2\bullet} = C/(1+C)$.
The product-form (\ref{eq:stationaryQOrapid}) can be formally
justified in the framework of singularly perturbed Markov chains \cite{Altmanetal2004,Avrachenkovetal2013}.

Using (\ref{eq:si}), we conclude that
$$
s_i(0) \to \frac{1+C}{\mu_1+C\mu_2}=\frac{1}{\bar{\mu}}, \quad i=1,2,
$$
as $\eta_1,\eta_2 \to \infty$. And hence,
\begin{equation}
\label{eq:snrapid}
s(n)=\sum_{k=0}^{n} Q^k s(0) = \sum_{k=0}^{n} Q^k 
\left[\begin{array}{l}1\\1\end{array}\right] \frac{1}{\bar{\mu}}
= \frac{(n+1)}{\bar{\mu}} \left[\begin{array}{l}1\\1\end{array}\right].
\end{equation}
Substituting (\ref{eq:stationaryQOrapid}) and (\ref{eq:snrapid}) into (\ref{eq:U_QOcase}),
yields
$$
U_{L=n} = R - c_w \frac{n+1}{\bar{\mu}},
$$
which means that this case corresponds to the Naor's classical model with the averaged service
rate $\bar{\mu}$.

\subsection{QO case: very slow transitions between the environment phases}

Now, let us consider the opposite situation when the transition rates between the phases of the
environment are very slow. Technically, this means to let the rates $\eta_1$ and $\eta_2$ go to
zero assuming that their ratio $\eta_1/\eta_2$ goes to a finite constant $C$. 

In this case, as $\eta_1$ and $\eta_2$ go to zero,  
the stationary distribution converges to the product-form (the justification again follows from the
theory of singularly perturbed Markov chains \cite{Altmanetal2004,Avrachenkovetal2013}), i.e.,
$$
p_{in}(a) \propto p_{i\bullet} \rho_i^n \prod_{k=0}^{n-1}a_k.
$$
so $p^a_{1,n+1}p^a_{2n}-p^a_{1n}p^a_{2,n+1}$ has the sign of $\rho_1-\rho_2$.
Thus, in this limiting case, Proposition~\ref{prop:QO-monotonicity} gives the primitive sufficient condition $(\mu_1-\mu_2)(\rho_1-\rho_2)\le 0$ for equilibrium threshold structure. Thus, the threshold structure is guaranteed when the slower service environment has the larger traffic intensity. In this case, the informational and congestion effects act 
in the same direction.

Specifically, using \cite{Altmanetal2004}, for a threshold strategy $\bar{n}$ we can write
\begin{equation}
\label{eq:stationaryQOslow}
p_{in} = p_{i\bullet} \frac{1-\rho_i}{1-\rho_i^{\bar{n}+1}} \rho_i^n,
\quad n=0,...,\bar{n},
\end{equation}
where $\rho_i=\lambda_i/\mu_i$, $i=1,2$.

From the expression (\ref{eq:si}), we conclude that
$$
s_i(0) \to \frac{1}{\mu_i}, \quad i=1,2,
$$
as $\eta_1,\eta_2 \to 0$. From (\ref{eq:qii}), we conclude that
$Q \to I$ (identity matrix), as $\eta_1,\eta_2 \to 0$. Hence, using (\ref{eq:sn_1}),
\begin{equation}
\label{eq:snslow}
s(n)=\sum_{k=0}^n Q^k s(0)=
(n+1)\left[\begin{array}{l}1/\mu_1\\1/\mu_2\end{array}\right].
\end{equation}
Thus, by (\ref{eq:U_QOcase}), the utility against
everyone following a threshold policy is
$$
U_{L=n}^{(\bar{n})} = 
R - c_w (n+1) \left(\frac{\alpha^{\bar{n}}_n}{\mu_1}
+\frac{1-\alpha^{\bar{n}}_n}{\mu_2}\right),
$$
where $\alpha^{(\bar{n})}_n = \mathbb{P}_{\bar{n}}^0(E=1 | L=n)$ is
the conditional Palm probability when everyone follows the threshold
strategy $\bar{n}$.
Thus, the threshold strategy $\bar{n}$ is an equilibrium when
$$
\bar{n} \left(\frac{\alpha^{(\bar{n})}_{\bar{n}-1}}{\mu_1}
+\frac{1-\alpha^{(\bar{n})}_{\bar{n}-1}}{\mu_2}\right) \le \frac{R}{c_w},
$$
and
$$
(\bar{n}+1) \left(\frac{\alpha^{(\bar{n})}_{\bar{n}}}{\mu_1}
+\frac{1-\alpha^{(\bar{n})}_{\bar{n}}}{\mu_2}\right) > \frac{R}{c_w}.
$$

If $\mu_1=\mu_2$, then we again retrieve Naor's model. Another interesting particular case 
is when $\rho_1=\rho_2$ and we have
$$
U_{L=n} = 
R - c_w\frac{\lambda_1 p_{1\bullet}/\mu_1 + \lambda_2 p_{2\bullet}/\mu_2}{\lambda_1 p_{1\bullet}
+\lambda_2 p_{2\bullet}} (n+1) \ge 0,
$$ 
which corresponds to the Naor's model with a modified mean service time, where the expectation
is conditioned on an arrival event.

\section{EO case: only the environment is observable}
\label{sec:EOcase}

Now, we consider the EO case when an arriving customer observes the current phase of the environment
but does not see the queue length. Then, the customers' equilibrium strategy can be described by two probabilities
$a_1$ and $a_2$. It is a randomized strategy \cite{Edelson1975,HassinHaviv}. A customer decides to join 
the queue with probability $a_i$, if they observe the environment in phase $i=1,2$. Under this strategy, the dynamics
of the system is a random walk in random environment (for the state transitions see Fig.~\ref{fig:trans_diag1}).

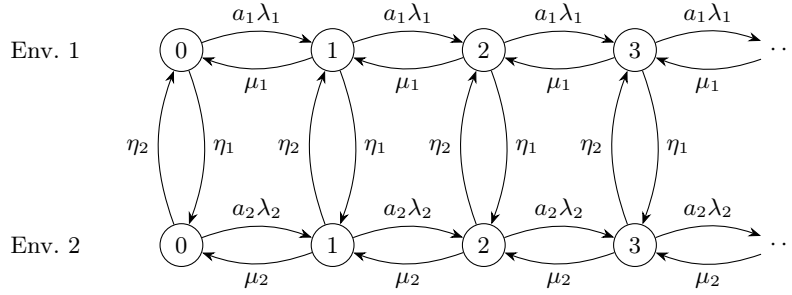
\begin{figure}[ht]
\centering
\begin{tikzpicture}[->, >=Stealth, node distance=2cm]

% States for Environment 1 (upper chain)
\node[circle, draw] (s0) {0};
\node[circle, draw, right of=s0] (s1) {1};
\node[circle, draw, right of=s1] (s2) {2};
\node[circle, draw, right of=s2] (s3) {3};
\node[right of=s3] (dots) {$\cdots$};

% Birth-death arrows for Env 1
\draw[->] (s0) to[bend left=20] node[above] {$a_1\lambda_1$} (s1);
\draw[->] (s1) to[bend left=20] node[below] {$\mu_1$} (s0);
\draw[->] (s1) to[bend left=20] node[above] {$a_1\lambda_1$} (s2);
\draw[->] (s2) to[bend left=20] node[below] {$\mu_1$} (s1);
\draw[->] (s2) to[bend left=20] node[above] {$a_1\lambda_1$} (s3);
\draw[->] (s3) to[bend left=20] node[below] {$\mu_1$} (s2);
\draw[->] (s3) to[bend left=20] node[above] {$a_1\lambda_1$} (dots);
\draw[->] (dots) to[bend left=20] node[below] {$\mu_1$} (s3);

% States for Environment 2 (lower chain)
\foreach \i/\name in {0/s0, 1/s1, 2/s2, 3/s3} {
  \node[circle, draw, below=2.0cm of \name] (s\i b) {\i};
}
\node[right of=s3b] (dotsb) {$\cdots$};

% Birth-death arrows for Env 2
\draw[->] (s0b) to[bend left=20] node[above] {$a_2\lambda_2$} (s1b);
\draw[->] (s1b) to[bend left=20] node[below] {$\mu_2$} (s0b);
\draw[->] (s1b) to[bend left=20] node[above] {$a_2\lambda_2$} (s2b);
\draw[->] (s2b) to[bend left=20] node[below] {$\mu_2$} (s1b);
\draw[->] (s2b) to[bend left=20] node[above] {$a_2\lambda_2$} (s3b);
\draw[->] (s3b) to[bend left=20] node[below] {$\mu_2$} (s2b);
\draw[->] (s3b) to[bend left=20] node[above] {$a_2\lambda_2$} (dotsb);
\draw[->] (dotsb) to[bend left=20] node[below] {$\mu_2$} (s3b);

% Switching arrows between environments
\foreach \i in {0,1,2,3} {
  % From Env 1 to Env 2 at rate eta_1 (bend left)
  \draw[->] (s\i) to[bend left=20] node[right, pos=0.5] {$\eta_1$} (s\i b);
  % From Env 2 to Env 1 at rate eta_2 (bend left)
  \draw[->] (s\i b) to[bend left=20] node[left, pos=0.5] {$\eta_2$} (s\i);
}

% Labels
\node[left of=s0, xshift=0.2cm] {Env. 1};
\node[left of=s0b, xshift=0.2cm] {Env. 2};

\end{tikzpicture}
\caption{Transition diagram when only the environment is observable.}
\label{fig:trans_diag1}
\end{figure}

The expected utility can be written in the following form:
$$
U_{E=i} = R - c_w \sum_{n=0}^{\infty} \E^0[S|L=n,E=i] \Prob^0[L=n|E=i]
$$
\begin{equation}
\label{ExpUtilEO}
= R -c_w \sum_{n=0}^\infty s_i(n) \frac{p^0_{in}}{p^0_{i\bullet}}
= R -c_w \sum_{n=0}^\infty s_i(n) \frac{p_{in}}{p_{i\bullet}},
\quad i=1,2.
\end{equation}
We observe that, in this case, the factors involving the arrival rates cancel
in (\ref{ExpUtilEO}), so that we can use the stationary probabilities rather than their Palm counterparts. This is consistent with the literature \cite{VanDoornRegterschot1988,MelamedWhitt1990}, 
where it was established that the PASTA property holds conditionally on the environment. 

The equilibrium probabilities $a_i^*, i=1,2,$ are solutions of the best-response
conditions
\begin{equation}
\label{eq:equilibrium}
\left\{\begin{array}{ll}
U_{E=i}(a_1^*,a_2^*) \le 0, & a_i^*=0,\\
U_{E=i}(a_1^*,a_2^*) = 0, & 0<a_i^*<1,\\
U_{E=i}(a_1^*,a_2^*) \ge 0, & a_i^*=1.
\end{array}\right.
\end{equation}

%The equilibrium probabilities $a_i, i=1,2,$ are solutions of the inequalities
%\begin{equation}
%\label{eq:equilibrium}
%\left\{\begin{array}{l}
%U_{E=1}(a_1,a_2) \ge 0,\\
%U_{E=2}(a_1,a_2) \ge 0.
%\end{array}\right.
%\end{equation}
Again, we have an explicit expression for $s_i(n)$ but not for the
stationary probabilities $p_{in}$. Therefore, as in the QO case, we consider 
two limiting cases: the case of rapid oscillations of the environment and
the case of very slow transitions between the environment phases. In these 
cases, the stationary distribution has an explicit product-form.

\subsection{EO case: rapid oscillations of the environment}

Here we assume that the environment oscillates rapidly between two phases 1 and 2.
Specifically, $\eta_1$ and $\eta_2$ tend to infinity simultaneously and their
ratio $\eta_1/\eta_2$ tends to a finite constant $C$. In the limit, as the rates
go to infinity, the stationary probabilities converge to
\begin{equation}
\label{eq:stationaryC}
p_{in} = p_{i\bullet} (1-\bar{\lambda}/\bar{\mu}) (\bar{\lambda}/\bar{\mu})^n,
\end{equation}
where
$$
\bar{\lambda} = p_{1\bullet} a_1 \lambda_1 + p_{2\bullet} a_2 \lambda_2
\quad
\bar{\mu} = p_{1\bullet} \mu_1 + p_{2\bullet} \mu_2,
$$
and
$p_{1\bullet} = 1/(1+C)$, $p_{2\bullet} = C/(1+C)$.
In addition to the justification in \cite{YechialiNaor1971}, this can also be
justified in the framework of singularly perturbed Markov chains 
\cite{Altmanetal2004,Avrachenkovetal2013}.

Using (\ref{eq:snrapid}), we conclude that for $i=1,2$
$$
U_{E=i}(a_1,a_2) = 
R - \frac{c_w}{\bar{\mu}} \left(1-\frac{\bar{\lambda}}{\bar{\mu}}\right) 
\sum_{n=0}^\infty (n+1) \left(\frac{\bar{\lambda}}{\bar{\mu}}\right)^n
= R - \frac{c_w}{\bar{\mu}-\bar{\lambda}}
$$
and the utility is independent of the phase when a customer arrives.
Intuitively, this follows from the fact that the environment oscillates
very rapidly. The system as a whole resembles an M/M/1 queue with arrival rate 
$\bar{\lambda}$, service rate $\bar{\mu}$, whose mean sojourn time is
$1/(\bar{\mu}-\bar{\lambda})$.

Define $\lambda_{max}=p_{1\bullet} \lambda_1 + p_{2\bullet} \lambda_2$.
Similarly to Chapter 3 of \cite{HassinHaviv}, we can classify the equilibrium
strategies into three cases:
\begin{enumerate}
\item If $\lambda_{max} \le \bar{\mu} - c_w/R$, then $a_1^*=a_2^*=1$;
\item If $0 \le \bar{\mu}-c_w/R < \bar{\lambda}_{max}$, 
choose $a_1^*,a_2^*$ such that 
\begin{equation}
\label{eq:multi_equil}
p_{1\bullet} \lambda_1 a_1 + p_{2\bullet} \lambda_2 a_2 = \bar{\mu}-c_w/R;
\end{equation}
\item If $\bar{\mu}-c_w/R < 0$ (i.e., $R<c_w/\bar{\mu}$), then $a_1^*=a_2^*=0$.
\end{enumerate}
We note that there is freedom in the choice of $a_1^*$ and $a_2^*$ in case 2.
One natural option is to choose the same $a^*=(\bar{\mu}-c_w/R)/\bar{\lambda}_{max}$ for
the two phases of the environment.

\subsection{EO case: very slow transitions between the environment phases}

Now, let us consider the opposite situation when the transition rates between the phases of the
environment are very slow. Technically, this means to let the rates $\eta_1$ and $\eta_2$ go to
zero assuming that their ratio $\eta_1/\eta_2$ goes to a finite constant $C$. 

In this case, as $\eta_1$ and $\eta_2$ go to zero,  
the stationary distribution converges to the product-form \cite{YechialiNaor1971}:
\begin{equation}
\label{eq:stationaryD}
p_{in} = p_{i\bullet} (1-a_i\lambda_i/\mu_i) (a_i\lambda_i/\mu_i)^n, \quad i=1,2.
\end{equation}
Due to strategic customers, we can assume that $a_i \lambda_i < \mu_i$, $i=1,2$,
which implies that the system is stable in equilibrium.
(The justification of the product form can be provided using the theory of singularly perturbed 
Markov chains \cite{Altmanetal2004,Avrachenkovetal2013}.)

Combining (\ref{ExpUtilEO}) with (\ref{eq:snslow}) and (\ref{eq:stationaryD}),
we obtain
$$
U_{E=i} = R - \frac{c_w}{\mu_i-a_i\lambda_i}.
$$ 
Thus, in the case of very slow transitions between the environment phases,
we have a combination of two models from \cite{Edelson1975}, and the 
equilibrium strategies are given by
\begin{equation}
\label{eq:slow_limit}
a_i^*=\left[\frac{\mu_i-c_w/R}{\lambda_i}\right]_{[0,1]}, \quad i=1,2,
\end{equation}
where $[x]_{[0,1]}:=\min\{1,\max\{0,x\}\}$.

%This leads to the following classification of the equilibrium strategies:
%\begin{enumerate}
%\item If $\lambda_i \le \mu_i - c_w/R$, $i=1,2$, then $a_1^*=a_2^*=1$;
%\item If  $\lambda_i \le \mu_i - c_w/R$ and $0 \le \mu_j-c_w/R < \lambda_j$, $j\neq i$,
%then $a_i^*=1$ and $a_j^*=(\mu_j-c_w/R)/\lambda_j$;
%\item If $0 \le \mu_i-c_w/R < \lambda_i$, $i=1,2$, then
%$a_i^*=(\mu_i-c_w/R)/\lambda_i$, $i=1,2$;
%\item If $\mu_i-c_w/R < 0$ and $0 \le \mu_j-c_w/R < \lambda_j$, $j\neq i$,
%then $a_i^*=0$ and $a_j^*=(\mu_j-c_w/R)/\lambda_j$;
%\item If $\mu_i-c_w/R < 0$, $i=1,2$, then $a_1^*=a_2^*=0$.
%\end{enumerate}
%In fact, in the case of very slow transitions between the environment phases,
%we have a combination of two models from \cite{Edelson1975}. 

\section{NO case: fully unobservable system}

Finally, let us consider the NO case when no information is available to an arriving customer except
the knowledge of the system parameters. In this case, the equilibrium strategy is again randomized \cite{HassinHaviv},
given by the joining probability $a$. Under this strategy, the dynamics
of the system is again a random walk in random environment (for the state transitions see Fig.~\ref{fig:trans_diag2}).

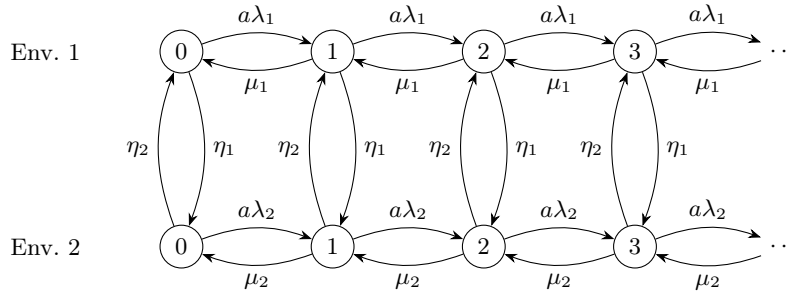
\begin{figure}[ht]
\centering
\begin{tikzpicture}[->, >=Stealth, node distance=2cm]

% States for Environment 1 (upper chain)
\node[circle, draw] (s0) {0};
\node[circle, draw, right of=s0] (s1) {1};
\node[circle, draw, right of=s1] (s2) {2};
\node[circle, draw, right of=s2] (s3) {3};
\node[right of=s3] (dots) {$\cdots$};

% Birth-death arrows for Env 1
\draw[->] (s0) to[bend left=20] node[above] {$a\lambda_1$} (s1);
\draw[->] (s1) to[bend left=20] node[below] {$\mu_1$} (s0);
\draw[->] (s1) to[bend left=20] node[above] {$a\lambda_1$} (s2);
\draw[->] (s2) to[bend left=20] node[below] {$\mu_1$} (s1);
\draw[->] (s2) to[bend left=20] node[above] {$a\lambda_1$} (s3);
\draw[->] (s3) to[bend left=20] node[below] {$\mu_1$} (s2);
\draw[->] (s3) to[bend left=20] node[above] {$a\lambda_1$} (dots);
\draw[->] (dots) to[bend left=20] node[below] {$\mu_1$} (s3);

% States for Environment 2 (lower chain)
\foreach \i/\name in {0/s0, 1/s1, 2/s2, 3/s3} {
  \node[circle, draw, below=2.0cm of \name] (s\i b) {\i};
}
\node[right of=s3b] (dotsb) {$\cdots$};

% Birth-death arrows for Env 2
\draw[->] (s0b) to[bend left=20] node[above] {$a\lambda_2$} (s1b);
\draw[->] (s1b) to[bend left=20] node[below] {$\mu_2$} (s0b);
\draw[->] (s1b) to[bend left=20] node[above] {$a\lambda_2$} (s2b);
\draw[->] (s2b) to[bend left=20] node[below] {$\mu_2$} (s1b);
\draw[->] (s2b) to[bend left=20] node[above] {$a\lambda_2$} (s3b);
\draw[->] (s3b) to[bend left=20] node[below] {$\mu_2$} (s2b);
\draw[->] (s3b) to[bend left=20] node[above] {$a\lambda_2$} (dotsb);
\draw[->] (dotsb) to[bend left=20] node[below] {$\mu_2$} (s3b);

% Switching arrows between environments
\foreach \i in {0,1,2,3} {
  % From Env 1 to Env 2 at rate eta_1 (bend left)
  \draw[->] (s\i) to[bend left=20] node[right, pos=0.5] {$\eta_1$} (s\i b);
  % From Env 2 to Env 1 at rate eta_2 (bend left)
  \draw[->] (s\i b) to[bend left=20] node[left, pos=0.5] {$\eta_2$} (s\i);
}

% Labels
\node[left of=s0, xshift=0.2cm] {Env. 1};
\node[left of=s0b, xshift=0.2cm] {Env. 2};

\end{tikzpicture}
\caption{Transition diagram for the fully unobservable case.}
\label{fig:trans_diag2}
\end{figure}

The expected utility is given by
\begin{equation}
\label{eq:ExpUtilNO}
U = R - c_w \sum_{n=0}^{\infty} \sum_{i=1,2} \E^0[S|E=i,L=n] \Prob^0[E=i,L=n],
\end{equation}
where the Palm distribution is given by
\begin{equation}
\label{eq:PalmDistr}
\Prob^0[E=i,L=n] = p^0_{in} = \frac{\lambda_i p_{in}}{\lambda_1p_{1\bullet}+\lambda_2p_{2\bullet}}.
\end{equation}
Thus, we have
\begin{equation}
\label{ExpUtilNOa}
U = R - c_w \left(\sum_{n=0}^\infty s_1(n) \frac{\lambda_1p_{1n}}{\lambda_1p_{1\bullet}+\lambda_2p_{2\bullet}} 
+ \sum_{n=0}^\infty s_2(n) \frac{\lambda_2p_{2n}}{\lambda_1p_{1\bullet}+\lambda_2p_{2\bullet}} \right).
\end{equation}
Since in general we do not have an explicit expression for the stationary
distribution, we proceed to the analysis of the two limiting cases. 

\subsection{NO case: rapid oscillations of the environment}

As in the EO scenario, in the limit as the rates
go to infinity, the stationary probabilities converge to \cite{YechialiNaor1971}
\begin{equation}
\label{eq:stationaryCNO}
p_{in} = p_{i\bullet} (1-\bar{\lambda}/\bar{\mu}) (\bar{\lambda}/\bar{\mu})^n,
\end{equation}
where
$$
\bar{\lambda} = a (p_{1\bullet} \lambda_1 + p_{2\bullet} \lambda_2)
\quad \mbox{and} \quad
\bar{\mu} = p_{1\bullet} \mu_1 + p_{2\bullet} \mu_2,
$$
$$
p_{1\bullet} = \frac{1}{1+C} \quad \mbox{and} \quad p_{2\bullet} = \frac{C}{1+C}.
$$
Combining (\ref{ExpUtilNOa}) with (\ref{eq:stationaryCNO}) and (\ref{eq:snrapid}), 
we obtain
$$
U = R -\frac{c_w}{\bar{\mu}-\bar{\lambda}}.
$$
Hence this case is equivalent to the model \cite{Edelson1975} with
the arrival and service rates averaged with respect to the rapidly
oscillating environment. Consequently, we have the following classification
of the equilibrium strategies:
\begin{enumerate}
\item If $\bar{\lambda}_{max} \le \bar{\mu} - c_w/R$, then $a^*=1$;
\item If $0 \le \bar{\mu}-c_w/R < \bar{\lambda}_{max}$, 
then $a^*=(\bar{\mu}-c_w/R)/\bar{\lambda}_{max}$;
\item If $\bar{\mu}-c_w/R < 0$ (i.e., $R<c_w/\bar{\mu}$), then $a^*=0$;
\end{enumerate}
where $\bar{\lambda}_{max}=p_{1\bullet} \lambda_1 + p_{2\bullet} \lambda_2
=\frac{\lambda_1 + C \lambda_2}{1+C}$.

\subsection{NO case: very slow transitions between the environment phases}

As in the EO case (see Section~4.2), in the limit as the rates
go to zero, the stationary probabilities converge to \cite{YechialiNaor1971}
\begin{equation}
\label{eq:stationaryDNO}
p_{in} = p_{i\bullet} (1-a\lambda_i/\mu_i) (a\lambda_i/\mu_i)^n, \quad i=1,2.
\end{equation}
%with
%$$
%p_{1\bullet} = \frac{1}{1+C} \quad \mbox{and} \quad p_{2\bullet} = \frac{C}{1+C}.
%$$
Combining (\ref{ExpUtilNOa}) with (\ref{eq:stationaryDNO}), (\ref{eq:PalmDistr}) and (\ref{eq:snslow}),
we can write
\begin{equation}
\label{ExpUtilNO1}
U = R - c_w\left(\frac{p^0_{1\bullet}}{\mu_1-a\lambda_1}+\frac{p^0_{2\bullet}}{\mu_2-a\lambda_2}\right)
= R - c_w \left(\frac{p^0_{1\bullet}/\mu_1}{1-a\rho_1}+\frac{p^0_{2\bullet}/\mu_2}{1-a\rho_2}\right),
\end{equation}
with $\rho_i=\lambda_i/\mu_i, i=1,2$. Without loss of generality, assume that $\rho_1 < \rho_2$.
Then, it is easy to see that the function
$$
f(a)=\left(\frac{p^0_{1\bullet}/\mu_1}{1-a\rho_1}+\frac{p^0_{2\bullet}/\mu_2}{1-a\rho_2}\right)
$$
has two vertical asymptotes at $a=1/\rho_2$ and $a=1/\rho_1$ and, except at these two points, it has positive
derivative. We need to consider only the interval $[0,1/\rho_2)$.
Note also that
$$
f(0)=\frac{p^0_{1\bullet}}{\mu_1}+\frac{p^0_{2\bullet}}{\mu_2}.
$$
Thus, we have the following classification of the equilibrium strategies:
\begin{enumerate}
\item If $\frac{R}{c_w} \le \frac{p^0_{1\bullet}}{\mu_1}+\frac{p^0_{2\bullet}}{\mu_2}$,
then $a^*=0$;
\item If $\rho_2 < 1$, there are two sub-cases:
\begin{enumerate}
\item If $\frac{R}{c_w} > \frac{p^0_{1\bullet}/\mu_1}{1-\rho_1}+\frac{p^0_{2\bullet}/\mu_2}{1-\rho_2}$,
then $a^*=1$;
\item If $\frac{R}{c_w} \le \frac{p^0_{1\bullet}/\mu_1}{1-\rho_1}+\frac{p^0_{2\bullet}/\mu_2}{1-\rho_2}$,
then 
\begin{equation}
\label{eq:astarNO}
a^*=\frac{(\rho_1+\rho_2)R/c_w-\rho_2p^0_{1\bullet}/\mu_1-\rho_1p^0_{2\bullet}/\mu_2-\sqrt{D}}{2\rho_1\rho_2R/c_w},
\end{equation}
where
$$
D=\left((\rho_1+\rho_2)\frac{R}{c_w}-\frac{\rho_2p^0_{1\bullet}}{\mu_1}-\frac{\rho_1p^0_{2\bullet}}{\mu_2}\right)^2
-4\rho_1\rho_2\frac{R}{c_w}\left(\frac{R}{c_w}-\frac{p^0_{1\bullet}}{\mu_1}-\frac{p^0_{2\bullet}}{\mu_2}\right).
$$
\end{enumerate}
\item If $\rho_2 \ge 1$ and $\frac{R}{c_w} > \frac{p^0_{1\bullet}}{\mu_1}+\frac{p^0_{2\bullet}}{\mu_2}$, 
then $a^*$ is given by (\ref{eq:astarNO}).
\end{enumerate}
It is interesting to observe that in comparison with the EO case with slow transitions between 
the environment phases, in the NO case we do not have a straightforward modification of the basic model in \cite{Edelson1975}.  

\section{Numerical example}

We further investigate numerically the EO case described in
Section~\ref{sec:EOcase}. As noted there, for finite values of
$\eta_1$ and $\eta_2$, we do not have a closed-form expression for the
stationary distribution. However, for fixed joining probabilities
$a_1$ and $a_2$, the process is a quasi-birth-and-death (QBD) process
with $2\times 2$ matrix blocks describing transitions between and within
levels. Thus, the stationary distribution has the matrix-geometric form
$$
p_{n}=p_{0} {\cal R}^n, \quad n=1,2,...
$$
where $p_{n}=[p_{1n} \ p_{2n}]$ and ${\cal R} \in {\mathbb R}^{2\times 2}$  
is the minimal nonnegative solution of a second-order polynomial matrix 
equation \cite{Neuts1978,Neuts1981}. 

Let us choose the following set of parameters: $\lambda_1=0.8$,
$\lambda_2=1.2$, $\mu_1=1.0$, $\mu_2=1.4$ and $c_w=1$, $R=5/3$.
We take $\eta_1=C\eta_2$ with $C=0.75$ and vary $\eta_2$.
With this value of $C$, the environment marginal stationary
probabilities are
$p_{1\bullet} = 1/(1+C)=4/7$,
$p_{2\bullet} = C/(1+C)=3/7$.
By (\ref{eq:slow_limit}), the slow-transition limit is
$$
[a_1^*, a_2^*] \xrightarrow[\eta_2 \to 0]{}
\left[\frac{1}{2},\frac{2}{3}\right].
$$
We plot the equilibrium joining probabilities $a_1^*$ and $a_2^*$
as functions of $\eta_2$ in Figure~\ref{fig1}. The dashed lines there
correspond to the slow-transition limit. 

In the fast-transition limit, as $\eta_1,\eta_2\to\infty$, any pair
$(a_1^*,a_2^*)$ satisfying~(\ref{eq:multi_equil}) constitutes an
equilibrium. In the present numerical example, this equation takes the
form
$$
8a_1^* + 9a_2^* = 10,
$$
which is satisfied to good accuracy by $a_1^*(100)=0.395$ and
$a_2^*(100)=0.759$. Note that, for every finite $\eta_2$, the 
two phase-dependent utilities
determine a particular pair $(a_1^*,a_2^*)$, whereas in the
rapid-oscillation limit the two utilities coincide and the limiting
model admits a continuum of equilibria. The numerical continuation
appears to select a point near $(0.395,0.759)$ from this continuum.

\begin{figure}
\centering
\includegraphics[width=0.8\textwidth]{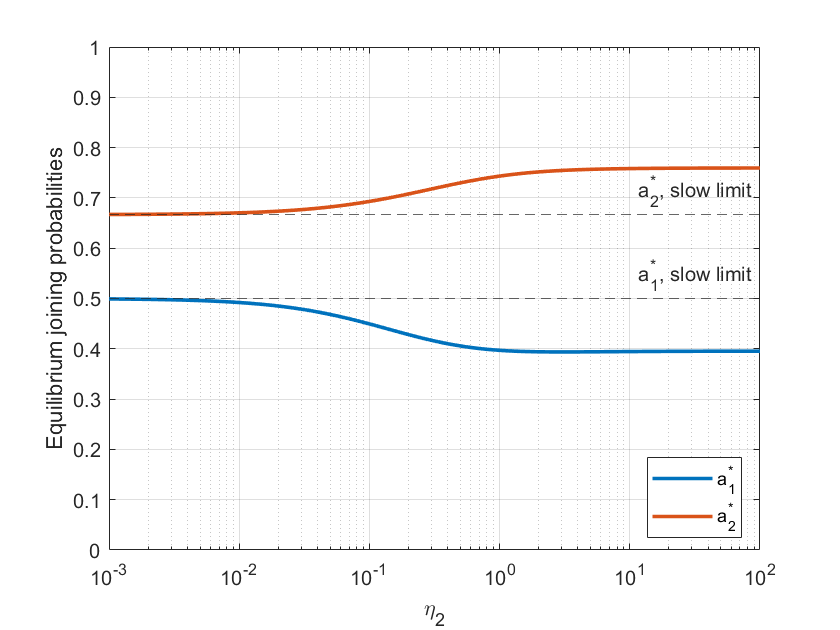}
\caption{EO case: Equilibrium joining probabilities $a_1^*$ and $a_2^*$
as functions of $\eta_2$. The dashed lines correspond to the limits
of the joining probabilities as $\eta_2 \to 0$.} 
\label{fig1}
\end{figure}

\section{Conclusion}

We studied equilibrium joining strategies in an M/M/1-type queue operating in a two-phase random environment under four information scenarios: fully observable (FO), queue-length observable (QO), environment observable (EO), and no observation (NO). We derived equilibrium characterizations in each case and obtained explicit results in the limiting regimes of rapid and very slow oscillations of the environment. 
Several interesting questions remain open. For example, we would like to understand equilibrium selection in singular limits, for example whether the equilibrium selected by continuation from finite transition rates can be characterized analytically or justified by another selection principle. We would also like 
to investigate monotonicity results for the EO case.

\begin{credits}
\subsubsection{\ackname}
We would like to thank Jake Clarkson and Refael Hassin for their helpful remarks.  
U. Yechiali is supported by the Israel Science Foundation, Grant No.1968/23.

%\subsubsection{\discintname}
%The authors have no competing interests to declare that are
%relevant to the content of this article.
\end{credits}


\begin{thebibliography}{99}

\bibitem{Altmanetal2004}
Altman, E., Avrachenkov, K.E., Núnez-Queija, R.:
Perturbation analysis for denumerable Markov chains with application to queueing models.
Advances in Applied Probability \textbf{36}(3), 839--853 (2004)

\bibitem{Avrachenkovetal2013}
Avrachenkov, K.E., Filar, J.A., Howlett, P.G.:
\emph{Analytic Perturbation Theory and its Applications}.
SIAM (2013)

\bibitem{BaccelliMakowski1986}
Baccelli, F., Makowski, A.M.:
Stability and bounds for single server queues in random environment.
Stochastic Models \textbf{2}(2), 281--291 (1986)

\bibitem{BaccelliBremaud1987}
Baccelli, F., Br\'emaud, P.:
\emph{Palm Probabilities and Stationary Queues}.
Springer (1987)

\bibitem{EconomouKanta2008}
Economou, A., Kanta, S.:
Equilibrium balking strategies in the observable single-server queue with breakdowns and repairs.
Operations Research Letters \textbf{36}(6), 696--699 (2008)

\bibitem{Economou2013}
Economou, A., Manou, A.:
Equilibrium balking strategies for a clearing queueing system in alternating environment.
Annals of Operations Research \textbf{208}, 489--514 (2013)

\bibitem{Economou2016}
Economou, A., Manou, A.:
Strategic behavior in an observable fluid queue with an alternating service process.
European Journal of Operational Research \textbf{254}(1), 148--160 (2016)

\bibitem{Economou2022}
Economou, A., Logothetis, D., Manou, A.:
The value of reneging for strategic customers in queueing systems with server vacations/failures.
European Journal of Operational Research \textbf{299}(3), 960--976 (2022)

\bibitem{Edelson1975}
Edelson, N.M., Hilderbrand, D.K.:
Congestion tolls for Poisson queuing processes.
Econometrica \textbf{43}, 81--92 (1975)

\bibitem{Gupta2006}
Gupta, V., Harchol-Balter, M., Wolf, A.S., Yechiali, U.:
Fundamental characteristics of queues with fluctuating load.
In: Proceedings of ACM SIGMETRICS, pp. 203--215 (2006)

\bibitem{HassinHaviv}
Hassin, R., Haviv, M.:
\emph{To Queue or Not to Queue: Equilibrium Behavior in Queueing Systems}.
Springer (2003)

\bibitem{Hassin2016}
Hassin, R.:
\emph{Rational Queueing}.
CRC Press (2016)

\bibitem{Hassinatal2023}
Hassin, R., Haviv, M., Oz, B.:
Strategic behavior in queues with arrival rate uncertainty.
European Journal of Operational Research \textbf{309}(1), 217--224 (2023)

\bibitem{Jagannathan2012}
Jagannathan, K., Menache, I., Modiano, E., Zussman, G.:
Non-cooperative spectrum access: The dedicated vs.\ free spectrum choice.
IEEE Journal on Selected Areas in Communications \textbf{30}(11), 2251--2261 (2012)

\bibitem{KemenySnell1976}
Kemeny, J.G., Snell, J.L.:
\emph{Finite Markov Chains}.
2nd edn. Springer (1976)

\bibitem{MelamedWhitt1990}
Melamed, B., Whitt, W.:
On arrivals that see time averages: A martingale approach.
Journal of Applied Probability \textbf{27}(2), 376--384 (1990)

\bibitem{MullerStoyan2002}
M\"uller, A., Stoyan, D.:
\emph{Comparison Methods for Stochastic Models and Risks}.
Wiley (2002)

\bibitem{Naor}
Naor, P.:
The regulation of queue size by levying tolls.
Econometrica \textbf{37}(1), 15--24 (1969)

\bibitem{Neuts1978}
Neuts, M.F.:
The M/M/1 queue with randomly varying arrival and service rates.
Opsearch \textbf{15}, 139--157 (1978)

\bibitem{Neuts1981}
Neuts, M.F.:
\emph{Matrix-Geometric Solutions in Stochastic Models: An Algorithmic Approach}.
Johns Hopkins University Press (1981)

\bibitem{Paz2007}
Paz, N., Yechiali, U.:
A note on the $M/M/\infty$ queue in random environment.
Technical report (2007).
\url{https://www.math.tau.ac.il/~uriy/Publications.html}

\bibitem{Paz2014}
Paz, N., Yechiali, U.:
An M/M/1 queue in random environment with disasters.
Asia-Pacific Journal of Operational Research \textbf{31}(3), 1450016 (2014)

\bibitem{Purdue1974}
Purdue, P.:
The M/M/1 queue in a Markovian environment.
Operations Research \textbf{22}(3), 562--569 (1974)

\bibitem{VanDoornRegterschot1988}
Van Doorn, E.A., Regterschot, G.J.K.:
Conditional PASTA.
Operations Research Letters \textbf{7}(5), 229--232 (1988)

\bibitem{YechialiNaor1971}
Yechiali, U., Naor, P.:
Queuing problems with heterogeneous arrivals and service.
Operations Research \textbf{19}(3), 722--734 (1971)

\bibitem{Yechiali1971}
Yechiali, U.:
On optimal balking rules and toll charges in the GI/M/1 queuing process.
Operations Research \textbf{19}(2), 349--370 (1971)

\bibitem{Yechiali1972}
Yechiali, U.:
Customers' optimal joining rules for the GI/M/s queue.
Management Science \textbf{18}(7), 434--443 (1972)

\end{thebibliography}
\end{document}